\documentclass[10pt,twoside]{siamltex}

\usepackage{amsfonts,epsfig}
\usepackage[toc,page]{appendix}

\usepackage[T1]{fontenc}

\usepackage{xy}
\input{xy}
\xyoption{all} \xyoption{rotate}

\newtheorem{remark}[theorem]{Remark}

\begin{document}

\bibliographystyle{plain}
\title{$16$-vertex graphs with automorphism groups $A_{4}$ and $A_{5}$ from icosahedron}

\author{
Peteris\ Daugulis\thanks{Institute of Life Sciences and
Technologies, Daugavpils University, Daugavpils, LV-5400, Latvia
(peteris.daugulis@du.lv). } }

\pagestyle{myheadings} \markboth{P.\ Daugulis}{$16$-vertex graphs
with automorphism groups $A_{4}$ and $A_{5}$ from icosahedron}
\maketitle

\begin{abstract} The article deals with the problem of finding vertex-minimal graphs with a given automorphism group.
We exhibit two undirected $16$-vertex graphs having automorphism
groups $A_{4}$ and $A_{5}$. It improves the Babai's bound for
$A_{4}$ and the graphical regular representation bound for
$A_{5}$. The graphs are constructed using projectivisation of the
vertex-face graph of icosahedron.

\end{abstract}


\section{Introduction}\

\subsection{Outline}

This article addresses a problem in graph representation theory of
finite groups - finding undirected graphs with a given
automorphism group and minimal number of vertices.

Denote by $\mu(G)$ the minimal number of vertices of undirected
graphs having automorphism group isomorphic to $G$,
$\mu(G)=\min\limits_{\Gamma: Aut(\Gamma)\simeq G}|V(\Gamma)|$. It
is known \cite{B1} that $\mu(G)\le 2|G|$, for any finite group $G$
which is not cyclic of order $3,4$ or $5$. See Babai \cite{B2} and
Cameron \cite{C}, for expositions of this area. There are groups
which admit a graphical regular representation, for such groups
$\mu(G)\le |G|$.  For some recent work see \cite{Da}, \cite{Da1},
\cite{G}.

For alternating groups $A_{n}$ $\mu(A_{n})$ is known for $n\ge
13$, see Liebeck \cite{L}. If $n\equiv 0\ or\ 1 (mod\ 4)$, then
$\mu(A_{n})=2^n-n-2$. Additionally, for $n\ge 5$ $A_{n}$ admits a
graphical regular representation, see \cite{W}. Thus for $A_{5}$
the best published estimate until now seemed to be $\mu(A_{5})\le
60$.

In this paper we exhibit graphs $\Gamma_{i}=(V,E_{i})$, $i\in
\{4,5\}$, such that $|V|=16$ and $Aut(\Gamma_{i})\simeq A_{i}$.
The graph $\Gamma_{5}$ (also denoted $\Pi_{I}$) is listed in
\cite{Co} together with order of its automorphism group. These
$\mu$ values are less than the Babai's bound for groups $A_{4}$
and $A_{5}$. For $A_{5}$ our graph has fewer vertices than the
graphical regular representation.
The new graphs are based on projectivisation of
vertex-face incidence relation of icosahedron.

\subsection{Notations}

We use standard notations for undirected graphs, see Diestel
\cite{D}. A bipartite graph $\Gamma$ with vertex partition sets
$V_{1}$ and $V_{2}$ is denoted as $\Gamma=(V_{1},V_{2},E)$.


Given a polyhedron $P$, we denote its vertex, edge and face sets
as $V=V(P)$, $E=E(P)$ and $F=F(P)$, respectively. We can think of
$P$ as the triple $(V,E,F)$.

If $S$ is a subset of $\mathbb{R}^{3}$ not containing the origin,
then its image under a projectivisation map to $P(\mathbb{R}^{3})$
is denoted by $\pi(S)$ or $[S]$, $[S]=\bigcup_{x\in S}[x]$.

\section{Main results}\

\subsection{Vertex-face graphs of polyhedra}\

\begin{definition} Let $P=(V,E,F)$ be a polyhedron. An undirected bipartite graph
$\Gamma_{P}=(V,F,I)$ is the
 \textbf{vertex-face graph of $P$} if $v\sim f$ iff $v\in V$, $f\in
 F$ and $v\in f$. In other words, $\Gamma_{P}$ corresponds to the vertex-face
incidence relation in $V\times F$.
\end{definition}

\begin{definition} Let $S=(V,E,F)$ be a centrally symmetric polyhedron. Let $S$ be positioned in
$\mathbb{R}^{3}$ so that its center is at $(0,0,0)$.
 We call the undirected
bipartite graph $\Pi_{S}=([V],[F],I_{p})$ \textbf{projective
vertex-face graph} if for any $v_{p}\in [V]$, $f_{p}\in [F]$ we
have $v_{p}\sim f_{p}$ iff $v\in f$ for some $v\in
\pi^{-1}(v_{p})$ and $f\in \pi^{-1}(f_{p})$.

\end{definition}

%

\subsection{Projective vertex-face graph of icosahedron and $A_{5}$}\

Let $I=(V,E,F)$ be a regular icosahedron. Denote by $Rot(I)\le
SO(3)$ the group of rotational symmetries of $I$ - rotations of
$\mathbb{R}^{3}$ preserving $V$ and $E$. It is known that
$Rot(I)\simeq A_{5}$. $\Pi_{I}$ is shown in Fig.1. We note that
$\Pi_{I}$ can be interpreted as the vertex-face graph of
hemi-icosahedron, see \cite{M2}.

\begin{center}
\epsfysize=45mm
    \epsfbox{fig_3_6.mps}

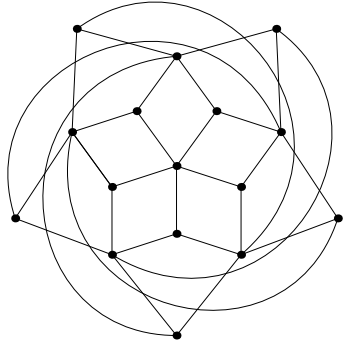
Fig.1.  - $\Pi_{I}$.
\end{center}

\begin{proposition}\label{1}
Let $I$ be regular icosahedron. Then $Aut(\Pi_{I})\simeq A_{5}$.
\end{proposition}

\begin{proof} The stated fact can be checked using an appropriate
software, such as Magma, see \cite{B3}. Nevertheless we give a
proof based on the geometric construction.
 We prove that
$Rot(I)\simeq Aut(\Pi_{I})$ in two steps.

First we prove that there is a subgroup in $Aut(\Pi_{I})$
isomorphic to $Rot(I)$. We show that there is an injective group
morphism $f:Rot(I)\stackrel{f_{1}}{\rightarrow}
Aut(\Gamma_{I})\stackrel{f_{2}}{\rightarrow}Aut(\Pi_{I}).$
$f_{1}:Rot(I)\rightarrow Aut(\Gamma_{I})$ maps every $\rho\in
Rot(I)$ to $f_{1}(\rho)\in Aut(\Gamma_{I})$ which is the
permutation of $V\cup F$ induced by $\rho$:
$f_{1}(\rho)(x)=\rho(x)$ for any $x\in V\cup F$. Rotations of $I$
preserve the vertex-face incidence relation and $f_{1}$ is a group
morphism. $f_{2}:Aut(\Gamma_{I})\rightarrow Aut(\Pi_{I})$ maps
every $\varphi\in Aut(\Gamma_{I})$ to $\varphi_{P}\in
Aut(\Pi_{I})$ defined by the rule $\varphi_{P}([x])=[\varphi(x)]$
for any $x\in V(\Gamma_{I})$. Projectivization and composition
commute therefore $f_{2}$ is a group morphism. $f$ is injective
since there is no nontrivial rotation of $I$ sending each vertex
to another vertex in the same projective class.

In the second step we show that $|Aut(\Pi_{I})|\le 60$ by a
counting argument. Every vertex $v\in [V]$ is contained in a
subgraph $\sigma(v)$ shown in Fig.2.

\begin{center}
\epsfysize=30mm
    \epsfbox{fig_1.mps}

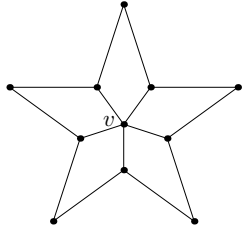
Fig.2.  - $\sigma(v)$.
\end{center}

All $\Pi_{I}$-vertices in $[V]$ have degree $5$, all
$\Pi_{I}$-vertices in $[F]$ have degree $3$. It follows that $[V]$
and $[F]$ both are unions of $Aut(\Pi_{I})$-orbits. $v$ can be
mapped by a $\Pi_{I}$-automorphism in at most $6$ possible ways.
After fixing the image of $v$ it follows again by
$Aut(\Pi_{I})$-invariance of $[V]$ that the subgraph $\sigma(v)$
can be mapped in at most $10$ ways. Any permutation of $[V]$ by an
automorphism determines a unique permutation of $[F]$. Thus
$|Aut(\Pi_{I})|\le 60$. We have shown that
$Aut(\Pi_{I})=f(Rot(I))\simeq A_{5}$.
\end{proof}

\begin{remark} A graph isomorphic to
$\Pi_{I}$ is listed without discussion of automorphism group in
\cite{Co} as one of connected edge-transitive bipartite graphs,
ET16.5.

\end{remark}

%

\subsection{A modification of the projective vertex-face graph of icosahedron and $A_{4}$}\

Since $A_{5}$ has subgroups isomorphic to $A_{4}$, we can try to
modify $\Pi_{I}$ so that the automorphism group of the modified
graph is isomorphic to $A_{4}$. We find generators for a subgroup
$H\le Rot(I)$, such that $H\simeq A_{4}$, and add $3$ extra edges
to $\Pi_{I}$ which are permuted only by elements of $H$.

Denote by $I_{1}$ the polyhedral ($1$-skeleton) graph of $I$,
$Aut(I_{1})\simeq Sym(I) \simeq A_{5}\times \mathbb{Z}_{2}$.
 An isomorphism
$Sym(I)\rightarrow Aut(I_{1})$ takes a symmetry $f$ to
$f|_{I_{1}}$.


\begin{proposition} Choose a $6$-subset of vertices $W=\{O,A,B,C,D,E\}\subseteq V(I)$
such that $I_{1}[W]$ is isomorphic to the $5$-wheel, see Fig.3.

\begin{center}
\epsfysize=25mm
    \epsfbox{fig_2.mps}

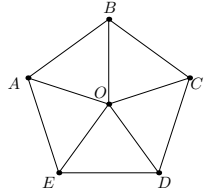
Fig.3.  - $I_{1}[W]$.
\end{center}

Define an undirected graph $\Xi_{I}=([V]\cup [F],I_{p}\cup J)$ by
adding $3$ edges to $\Pi_{I}$: $J=\{[A]\sim [C],[B]\sim
[O],[D]\sim [E]\}$, see Fig.4. Then $Aut(\Xi_{I})\simeq A_{4}$.

\begin{center}
\epsfysize=25mm
    \epsfbox{fig_4_4.mps}

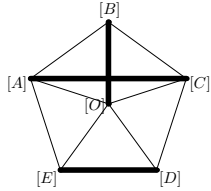
Fig.4.  - extra edges.
\end{center}

\end{proposition}

\begin{proof}
Consider the subgroup $H=\langle r_{1},r_{2}\rangle\le Rot(I)$
generated by two rotations:
 $r_{1}$ - rotation by $\frac{2\pi}{3}$ radians around the
line passing through the center of the face $OCD$ and the center
of $I$, $r_{2}$ - rotation by $\pi$ radians around the line
passing through the center of the edge $OB$ and the center of $I$.

It can be checked that $H\simeq A_{4}$. Note that the vertices
$O,A,B,C,D,E$ in Fig.3 represent the $6$ projective classes of
$V$.

We have to show that $Aut(\Xi_{I})\simeq H$. First we show that
$H\le Aut(\Xi_{I})$. $\Xi_{I}$ differs from $\Pi_{I}$ by $3$ extra
edges. It suffices to note by direct inspection that $r_{1}$
permutes these extra edges and $r_{2}$ fixes each of them. To show
that $Aut(\Xi_{I})\le H$ we observe that any additional rotation
$r'$ does not permute these three new edges and thus $r'\not\in
Aut(\Xi_{I})$.
\end{proof}

%
%
%

%
%

\begin{remark} If $D$ is dodecahedron then $\Pi_{D}\simeq
\Pi_{I}\simeq A_{5}$.

\end{remark}



\section*{Acknowledgements} We used $Magma$, see Bosma et al. \cite{B3}, and $Nauty$,
available at $http://cs.anu.edu.au/~bdm/data/$, see McKay and
Piperno \cite{M}. The author thanks Valentina Beinarovica for her
assistance.



\begin{thebibliography}{1}

\bibitem{B1}
L. Babai (1974),
\newblock {On the minimum order of graphs with given group},
\newblock {Canad. Math. Bull., 17, pp. 467-470}.

\bibitem{B2}
L. Babai (1995),
\newblock {Automorphism groups, isomorphism, reconstruction},
\newblock {In Graham, Ronald L.; Grotschel, Martin; Lovasz, Laszlo,
Handbook of Combinatorics I, North-Holland, pp. 1447-1540}.

\bibitem{B3}
W. Bosma, J. Cannon, and C. Playoust (1997),
\newblock {The Magma algebra system. I. The user language,}
\newblock {J. Symbolic Comput., 24, pp. 235-265.}

\bibitem{C}
P. Cameron (2004)
\newblock  {Automorphisms of graphs, in Topics in Algebraic Graph Theory (ed. L. W. Beineke and R. J.
Wilson),}
\newblock {Cambridge Univ. Press, Cambridge, (ISBN 0521801974),
pp.137-155.}

\bibitem{Co}
M. Conder (2017) {\em Complete list of all connected
edge-transitive bipartite graphs on up to 63 vertices} Retrieved
February 13, 2019, from https://www.math.auckland.ac.nz/~conder/
AllSmallETBgraphs-upto63-full.txt.

\bibitem{Da}
P. Daugulis  (2017),
\newblock {A note on another construction of graphs with 4n + 6 vertices and cyclic automorphism group of order 4n}
\newblock {Archivum Mathematicum, 53(1),  pp.13-18.}

\bibitem{Da1}
P. Daugulis (2014)
\newblock  {$10$-vertex graphs
with cyclic automorphism group of order $4$}.
\newblock {http://arxiv.org/abs/1410.1163}

\bibitem{D}
R. Diestel (2010),
\newblock  {Graph Theory}.
\newblock {Graduate Texts in Mathematics, Vol.173,
Springer-Verlag, Heidelberg}.

\bibitem{G}
C. Graves, S.J. Graves, L.-K.Lauderdale (2017),
\newblock {Smallest graphs with given generalized quaternion automorphism
group,}
\newblock {Journal of Graph Theory},  pp.1097-0118,
http://dx.doi.org/10.1002/jgt.22166

\bibitem{L}
M. Liebeck (1983),
\newblock  {On graphs whose full automorphism group is an alternating
group or a finite classical group,}
\newblock {Proc. London Math. Soc. (3) 47, 337-362.}

\bibitem{M}
B. D. McKay and A. Piperno (2013),
\newblock{Practical Graph
Isomorphism, II}, \newblock {J. Symbolic Computation, 60}, pp.
94-112.

\bibitem{M2}
P. McMullen and E. Schulte (2002),
\newblock  {6C. Projective Regular Polytopes}.
\newblock {Abstract Regular Polytopes (1st ed.)
Cambridge University Press, pp. 162-165, ISBN 0-521-81496-0}.

\bibitem{W}
M. E. Watkins (1974),
\newblock{Graphical regular representations of alternating, symmetric, and miscellaneous small groups},
\newblock {Aequationes mathematicae,  Volume 11, Issue 1}, pp 40-50.

\end{thebibliography}
\end{document}